\documentclass[11pt]{amsart}
\usepackage[utf8]{inputenc}
\usepackage[margin=0.75in]{geometry}
\usepackage{hyperref}
\usepackage{amsfonts}
\usepackage{amsmath}
\usepackage{amsthm}
\usepackage{amssymb}
\usepackage{mathrsfs}
\usepackage{tikz}
\usepackage[nameinlink,noabbrev]{cleveref}
\usepackage{amsaddr}
\usepackage{cleveref}

\DeclareFontFamily{OT1}{pzc}{}
\DeclareFontShape{OT1}{pzc}{m}{it}{<-> s * [1.200] pzcmi7t}{}
\DeclareMathAlphabet{\mathpzc}{OT1}{pzc}{m}{it}

\numberwithin{equation}{section}

\newcommand{\R}{\mathbb{R}}
\newcommand{\N}{\mathbb{N}}

\newcommand{\C}{\mathbb{C}}
\newcommand{\X}{\mathfrak{X}}
\newcommand{\D}{\mathcal{D}}

\renewcommand{\S}{\mathcal{S}}

\newcommand{\dom}{\operatorname{dom}}

\theoremstyle{plain}
\newtheorem{theorem}{Theorem}[section]
\newtheorem{lemma}[theorem]{Lemma}
\newtheorem{corollary}[theorem]{Corollary}

\theoremstyle{definition}
\newtheorem{definition}[theorem]{Definition}

\theoremstyle{remark}
\newtheorem{remark}{Remark}

\title[]{A Modern Functional Analytic Tour of the \\ Fourier--Bessel Transform}
\author{Cameron L. Williams}
\address{
Department of Mathematics \\
Embry-Riddle Aeronautical University \\
3700 Willow Creek Road \\
Prescott, AZ 86301
}

\begin{document}

\begin{abstract}
In this paper, we review the theory of the Fourier--Bessel transform. The Fourier--Bessel transform is an integral transform on the half-line that generalizes the Fourier cosine transform and is dependent upon a parameter $\nu > -1$. Many of the standard results for the Fourier--Bessel transform have become folklore results and are difficult to find single sources for or predate modern functional analysis and instead rely on hard analysis techniques that miss the elegance of more modern machinery. This review aims to centralize much of the basic theory while providing some new proofs and approaches that can be applied to other integral transforms beyond the Fourier and Fourier--Bessel transform. Particular care is taken to illustrate the distinctions in the theory of the Fourier--Bessel transform when $-1 < \nu < -\frac{1}{2}$ and $-\frac{1}{2} < \nu$.
\end{abstract}

\maketitle

\section{Introduction} \label{sec:1}

The Fourier--Bessel transform finds its origin in Hermann Hankel's 1875 paper \cite{Hankel1875} in which he established the invertibility of an integral transform that later would later bare his name, though is often referred to the Fourier--Bessel transform in the present-day mathematical literature. Hankel's motivation for studying these integrals was in the rigorous development of Fourier-like series based on Bessel functions due to such expansions' appearance when employing the method of separation of variables in solving partial differential equations in polar and spherical coordinate systems. Particularly, he wanted to make sense of series of the form
\begin{equation}
\sum_n a_n J_{\nu}(nx),
\end{equation}

\noindent where $J_{\nu}$ is the Bessel function of the first kind and solves the differential equation
\begin{equation*}
x^2 y'' + x y' + (x^2 - \nu^2) y = 0.
\end{equation*}

This is a second order linear differential equation, so we expect two linearly independent solutions to this differential equation. The differential equation can be simplified by first letting $y(x) = x^{\nu} z(x)$ to get
\begin{equation*}
x z'' + (2\nu+1)z' + x z = 0.
\end{equation*}

\noindent Expanding $z$ in the power series $z = \sum_n c_n x^n$ leads to $0 = c_1$ and a recurrence relation
\begin{equation*}
c_{n+2} = - \frac{1}{(n+2)(n+2\nu+2)} c_n
\end{equation*}

\noindent so that $0 = c_3 = c_5 = \ldots$ and $c_n$ is undefined if $\nu = -\frac{k}{2}$ for some integer $k > 1$. This suggests that $J_{\nu}$ is only simply defined for $-1 < \nu$ and must be analytically continued to account for $\nu \le -1$. After a judicious choice for $c_0$, $J_{\nu}$ is then given by
\begin{equation}
J_{\nu}(x) = \sum_{n=0}^{\infty} \frac{(-1)^n}{n! \Gamma(n+\nu+1)} \bigg(\frac{x}{2}\bigg)^{2n+\nu}
\end{equation}

\noindent and $\Gamma$ is the gamma function. For $\nu \neq 1, 2, 3, \ldots$, $J_{\nu}$ and $J_{-\nu}$ are clearly linearly independent solutions to the differential equation.

A careful analysis in the half-integer case shows that the singularity of the recurrence relation does not truly exist \cite[p. 41]{Watson}. Moreover, even though the gamma function is undefined on $\{\ldots, -3, -2, -1, 0\}$, its reciprocal is entire and so the series is well-defined for all $\nu$. However there is a complication: $J_{-n}(x) = (-1)^n J_n(x)$ for $n = 1, 2, 3, \ldots$ so that these are no longer linearly independent solutions. This can be remedied by employing the Bessel function of the second kind $Y_{\nu}$ \cite[\href{https://dlmf.nist.gov/10.2.E3}{(10.2.3)}]{NIST:DLMF}.

Hankel was especially interested in making sense of the equality
\begin{equation}
f(x) = \sum_n a_n J_{\nu}(nx),
\end{equation}

\noindent where the $a_n$ are given by the usual inner product formula, i.e. when a function is equal to its Fourier--Bessel series in some sense, similar to the case of Fourier's series. Similarly, he was interested in the analogues of Fourier's infinite integrals
\begin{align}
\int_0^{\infty}\cos(yz) \int_0^{\infty} \cos(ax) f(x)\,dx\,dy &= \frac{\pi}{2} f(z) \\
\int_0^{\infty}\sin(yz) \int_0^{\infty} \sin(ax) f(x)\,dx\,dy &= \frac{\pi}{2} f(z)
\end{align}

\noindent that would give rise to Fourier transform theory. These analogues would be the origin of the eponymous transform in the form of the integrals
\begin{equation}
\int_0^{\infty} y J_{\nu}(yz)\int_0^{\infty} x J_{\nu}(xy) f(x)\,dx\,dy = f(z).
\end{equation}

Proofs for many of the fundamental properties of the Fourier--Bessel transform have existed for over a century, including the identity above under appropriate assumptions, but many use hard real analysis techniques rather than softer functional analysis techniques as the theory predates modern functional analysis \cite[14.41]{Watson}. On the other hand, some of the more modern functional analysis-based treatments shortcut the finer details and leave many questions unanswered \cite[Ch. 9]{Akhiezer}. Much of the current mathematical work takes these facts for granted which has a particular challenge: the theory for the Fourier--Bessel transform is different when $-1 < \nu < -\frac{1}{2}$ and $-\frac{1}{2} < \nu$ as we shall see. Further, the literature does not always make it clear when the distinction is important. This disparity creates a disconnect and can be the cause of some confusion about the validity and applicability of theorems.

The discussion in Watson's treatise \cite[p. 576--577]{Watson} illustrates the difficulty with navigating the literature regarding Fourier--Bessel theory (even though it pertains to Fourier--Bessel series and not transform), especially as it pertains to the distinctions between the cases $-1 < \nu < -\frac{1}{2}$ and $-\frac{1}{2} < \nu$.

The Fourier--Bessel series for a sufficiently nice function $f$ on $[0,1]$ is a series of the form
$$ f(x) \sim \sum_{n=1}^{\infty} a_n J_{\nu}(j_{m,\nu}x), $$

\noindent where $j_{m,\nu}$ is the $m$th zero of the Bessel function $J_{\nu}$ and $a_n$ is given by
$$ a_n = \frac{2}{J_{\nu+1}^2(j_{m,\nu})} \int_0^1 t f(t) J_{\nu}(j_{m,\nu}t)\,dt. $$

\noindent Watson has the following quote discussing the convergence of such series as it pertains to the range of values that $\nu$ might take. Note that no reference is made to the work addressing $-1 < \nu < -\frac{1}{2}$.

\begin{quote}
Dini, however, remarked that he was unable to deal with the range $-1 < \nu < -\frac{1}{2}$, and limited himself to the range $\nu \ge -\frac{1}{2}$. Several subsequent writers, while proving theorems for the latter range, asserted that the extension to the former range was merely a matter of detail; but it was not until 1922 [after the first edition] that anyone took the trouble to supply the detail which is tedious and of no great interest.
\end{quote}

In contrast with the Fourier transform which has myriad excellent treatments across texts with varying perspectives, there is no central entry point for the functional analytic theory of the Fourier--Bessel transform. A further complication is that many of the treatments of the Fourier transform leave little room for generalization to other integral transforms as they often rely on the group-theoretic properties of the Fourier kernel to streamline the arguments.

In this survey article, we will focus on the functional analytic properties of the Fourier--Bessel transform. There are some original results about integral transforms and operators scattered throughout. A major focus of the survey will be to illuminate the distinctions between the $-1 < \nu < -\frac{1}{2}$ and $-\frac{1}{2} < \nu$ settings, the applicability of results, and ways to manage the challenges that appear when $-1 < \nu < -\frac{1}{2}$.

In \Cref{sec:2}, we will discuss the Hankel and Fourier--Bessel transforms and how they are related. In \Cref{sec:3}, we will establish the domain of the Fourier--Bessel transform for $-1 < \nu < -\frac{1}{2}$ and $-\frac{1}{2} < \nu$. In \Cref{sec:4}, we will explore properties of the Fourier--Bessel transform, including a dilation property and the Riemann--Lebesgue lemma. This is where the distinction between the $-1 < \nu < -\frac{1}{2}$ and $-\frac{1}{2} < \nu$ settings becomes most apparent which has major implications on the remainder of the paper. In \Cref{sec:5}, we present some new results regarding the $L^2$ properties of integral transforms, specifically $L^2$ norm-preservation. In \Cref{sec:6}, we establish eigenfunctions of the Fourier--Bessel transform. In \Cref{sec:7}, we develop the $L^2$ theory of the Fourier--Bessel transform by way of an $L^1\cap L^2$ analysis and a Schwartz space analysis, employing results from \Cref{sec:5,sec:6} to this end. While having both analyses is not necessary, there are merits to exploring both as there are marked differences from the Fourier case that can be instructive in the analysis of other integral transforms. In \Cref{sec:8}, we will discuss an uncertainty principle for the Fourier--Bessel transform and discuss its limitations compared to the usual Fourier uncertainty principle and remedies for these limitations.

\section{The Hankel and Fourier--Bessel Transforms} \label{sec:2}

In this section, we will briefly discuss the Hankel transform for historical purposes before shifting focus to the Fourier--Bessel transform. There is not a significant difference between the two as they are merely similarity transformations of each other. For the sake of keeping with the present-day mathematics perspective, we will then pivot to discussing only the Fourier--Bessel transform. We will define the Hankel and Fourier--Bessel transforms in the next two definitions. We will discuss restrictions on $\nu$ in the next section.

\begin{definition}
The \emph{Hankel transform of order $\nu$} for $\nu\in\R$, often denoted $\mathcal{H}_{\nu}$ in the literature, is defined as an integral transform on sufficiently nice functions $f$ on the positive half-line and $y\in \R^+$ by
\begin{equation}
\mathcal{H}_{\nu} f(y) = \int_{\R^+} k_{\nu}(xy) f(x)\,dx,
\end{equation}

\noindent where $k_{\nu}(x) = \sqrt{x} J_{\nu}(x)$.
\end{definition}

A close inspection of this definition shows similarity to Hankel's repeated integrals except rather than treating $x$ and $y$ separately in the integral, a $\sqrt{xy}$ component appears in both integrals to make the integral transform and its (formal) inverse more symmetric in $x$ and $y$.

\begin{definition}
The \emph{Fourier--Bessel transform of order $\nu$} for $\nu\in\R$, often denoted $\mathcal{F}_{\nu}$ in the literature, is defined as an integral transform on sufficiently nice functions $f$ on the positive half-line and $y\in \R^+$ by
\begin{equation}
\mathcal{F}_{\nu} f(y) = \int_{\R^+} \varphi_{\nu}(xy) f(x) x^{2\nu+1} \,dx,
\end{equation}

\noindent where $\varphi_{\nu}(x) = x^{-\nu} J_{\nu}(x)$.
\end{definition}

A quick computation shows that $k_{\nu}$ solves the differential equation $-k_{\nu}''(x) + \frac{\nu^2-\frac{1}{4}}{x^2} k_{\nu}(x) = k_{\nu}$ and $\varphi_{\nu}$ solves the differential equation $-\varphi_{\nu}''(x) - \frac{2\nu+1}{x}\varphi_{\nu}'(x) = \varphi_{\nu}(x)$. We will save the discussion for which functions classify as ``sufficiently nice'' for later, but it is easy to see that $C_c^{\infty}(\R^+)$ functions are admissible. Immediately we can see similarities and differences between the two definitions. Both integral transforms unsurprisingly involve the Bessel function $J_{\nu}$, both integral kernels $k_{\nu}$ and $\varphi_{\nu}$ treat $x$ and $y$ on equal footing (and moreover, they are both functions of $xy$ which will be important later), and both $k_{\nu}$ and $\varphi_{\nu}$ are eigenfunctions of a second order differential operator. There are two marked differences between the two definitions: the measure in the Hankel transform case is the usual Lebesgue measure $dx$, whereas the measure in the Fourier--Bessel transform case is $x^{2\nu+1}\,dx$; and $\sqrt{x}$ appears in $k_{\nu}(x)$ whereas $x^{-\nu}$ appears in $\varphi_{\nu}(x)$. $k_{\nu}$ is not an entire function except when $\nu$ is a half-integer greater than or equal to $-\frac{1}{2}$, whereas $\varphi_{\nu}$ is entire for all $\nu > -1$.

The Hankel transform and Fourier--Bessel transform definitions are in fact exactly equivalent when $\nu = -\frac{1}{2}$. Specifically, $J_{-1/2}(x) = \sqrt{\frac{2}{\pi x}} \cos(x)$ so that $k_{-1/2}(x) = \varphi_{-1/2}(x) = \sqrt{\frac{2}{\pi}} \cos(x)$ and the Fourier cosine transform emerges. As such, we will omit discussion of $\nu = -\frac{1}{2}$, and instead focus on $\nu \neq -\frac{1}{2}$ as the proofs require more care in these settings. The case of $\nu = -\frac{1}{2}$ can be recovered easily.

Beyond the $\nu = -\frac{1}{2}$ case, the Hankel and Fourier--Bessel transforms are similarity transformations of each other via the map $S_{\nu} f(x) = x^{\nu+1/2} f(x)$ for sufficiently nice functions $f$. Thus, instead of studying both the Hankel transform and the Fourier--Bessel transform, we have the freedom to choose one and port results over to the other by way of the similarity transformation $S_{\nu}$. Henceforth, we will focus our attention on the Fourier--Bessel transform $\mathcal{F}_{\nu}$.

An alternative approach to deriving the Hankel transform than by inspection of Hankel's repeated integral is that $\mathcal{H}_0$ can be established by way of the Fourier transform of radial functions (i.e. functions satisfying $f(\vec{r}) = f(|\vec{r}|)$) in two dimensions. Likewise, modifications of $\mathcal{H}_n$ can be established by proceeding similarly for three and higher dimensions, though this is limited to integer orders.

Perhaps a more palatable functional analytic approach is in the spectral theory of the Bessel differential operator $\Delta_{\nu}$ defined on $C^2(\R^+)$ by
\begin{equation}
\Delta_{\nu} = \frac{d^2}{dx^2} + \frac{2\nu+1}{x} \frac{d}{dx}
\end{equation}

\noindent on an $L^2$ space on the half-line. Note that $\Delta_{\nu} \varphi_{\nu} = -\varphi_{\nu}$. $\varphi_{\nu}$ then appears naturally in the spectral analysis of $\Delta_{\nu}$ as it is an eigenfunction of it. Rather than getting into the weeds on this point, we will simply take the Fourier--Bessel transform as a definition. While the spectral theory perspective is a very elegant way to think about the Fourier--Bessel transform---and many other integral transforms---such approaches are $L^2$ focused, so the integral transform and $L^1$ theory are not immediate.

\section{The Domain of the Fourier--Bessel Transform} \label{sec:3}

In the previous section, we postponed discussion upon which functions the Fourier--Bessel transform is defined. We will now rectify this. Prior to that, we must discuss some properties of the kernel $\varphi_{\nu}$.

\begin{lemma} \label{lem:asymptotics_of_jnu}
Let $x\in\R^+$, then we have the following bound for $\varphi_{\nu}$
\begin{equation}
|\varphi_{\nu}(x)| \le M_1 + M_2 x^{-\nu-1/2}
\end{equation}

\noindent for some $M_1, M_2 > 0$. Or, more precisely, there exists an $R > 0$ such that $|\varphi_{\nu}(x)| \le M_1$ for $x\in (0, R)$ and $|\varphi_{\nu}(x)| \le M_2 x^{-\nu-1/2}$ for $x\in (R, \infty)$.
\end{lemma}

\begin{proof}
This follows from the asymptotics of the Bessel function $J_{\nu}$ \cite[7.21.1]{Watson}:
\begin{align*}
J_{\nu}(x) \sim \sqrt{\frac{2}{\pi x}} \bigg(& \cos\bigg(x - \frac{\pi}{2} \nu - \frac{\pi}{4}\bigg) \sum_{n=0}^{\infty} (-1)^n \frac{\Gamma\big(2n + \frac{1}{2} - \nu\big)\Gamma\big(2n + \frac{1}{2} + \nu\big)}{2^{2n} (2n)! \Gamma\big(\frac{1}{2} - \nu\big) \Gamma\big(\frac{1}{2} + \nu\big)} \frac{1}{x^{2n}} \\
& - \sin\bigg(x - \frac{\pi}{2} \nu - \frac{\pi}{4}\bigg) \sum_{n=0}^{\infty} \frac{\Gamma\big(2n + \frac{3}{2} - \nu\big)\Gamma\big(2n + \frac{3}{2} + \nu\big)}{2^{2n+1} (2n+1)! \Gamma\big(\frac{1}{2} - \nu\big) \Gamma\big(\frac{1}{2} + \nu\big)} \frac{1}{x^{2n+1}}\bigg)
\end{align*}

\noindent Thus, the dominating behavior is $x^{-1/2}$ at infinity. As $\varphi_{\nu}(x) = x^{-\nu} J_{\nu}(x)$, we have the following asymptotic bound for $\varphi_{\nu}$:
\begin{equation}
|\varphi_{\nu}(x)| \le M_2 x^{-\nu-1/2}.
\end{equation}

While this is true asymptotically, this bound may not be true near $0$. The entirety of $\varphi_{\nu}$ implies that it is bounded near $0$, i.e. $|\varphi_{\nu}(x)| \le M_1$ near $0$. Stitching these two behaviors together gives the bound
\begin{equation}
|\varphi_{\nu}(x)| \le M_1 + M_2 x^{-\nu-1/2}.
\end{equation}

\noindent We can pick $M_1$ large enough to ensure that $|\varphi_{\nu}(x)| \le M_1$ on a large enough interval $(0,R)$ until the asymptotic behavior for large $x$ takes over. We can without loss of generality assume that $R > 1$.
\end{proof}

\begin{corollary} \label{cor:phin_bounded}
If $-\frac{1}{2} < \nu$, $\varphi_{\nu}$ is bounded on $\R^+$ and furthermore $\displaystyle \|\varphi_{\nu}\|_{\infty} = \varphi_{\nu}(0) = \frac{1}{2^{\nu} \Gamma(\nu+1)}$.
\end{corollary}

\begin{proof}
This bound relies on an integral representation for the Bessel function $J_{\nu}$ for positive $x$ and $-\frac{1}{2} < \nu$ \cite[3.3.3]{Watson}:
\begin{equation*}
J_{\nu}(x) = \frac{x^{\nu}}{2^{\nu} \Gamma\big(\nu+\frac{1}{2}\big) \sqrt{\pi}} \int_{-1}^1 (1-t^2)^{\nu-1/2} \cos(xt)\,dt.
\end{equation*}

\noindent From this, we have the following bound on $\varphi_{\nu}$:
\begin{equation}
|\varphi_{\nu}(x)| \le \frac{1}{2^{\nu} \Gamma\big(\nu+\frac{1}{2}\big) \sqrt{\pi}} \int_{-1}^1 (1-t^2)^{\nu-1/2}\,dt = \frac{1}{2^{\nu}\Gamma\big(\nu+\frac{1}{2}\big)\sqrt{\pi}} \frac{\sqrt{\pi}\Gamma\big(\nu+\frac{1}{2}\big)}{\Gamma(\nu+1)} = \frac{1}{2^{\nu} \Gamma(\nu+1)} = \varphi_{\nu}(0).
\end{equation}
\end{proof}

We can also show that $\varphi_{\nu}'$ is a bounded function on $\R^+$ when $-\frac{1}{2} < \nu$, though the above methodology does not directly apply as $\varphi_{\nu}'(x) = x^{-\nu} J_{\nu+1}(x)$ which does not allow for complete cancellation of the powers of $x$. However, from our above analysis, $x^{-\nu} J_{\nu}(x) \sim x^{-\nu-1/2}$ for large $x$, so $\varphi_{\nu}'$ is bounded on $(R,\infty)$. Boundedness on $[0,R]$ follows from the entirety of $\varphi_{\nu}'$. These results are critical to understanding on which functions the Fourier--Bessel transform can be defined as an integral transform and the differential structure of the Fourier--Bessel transform.

\begin{theorem} \label{thm:domain_of_fb}
For $\nu > -1$, $\mathcal{F}_{\nu}$ can be defined on $L^1((0,R), x^{2\nu+1}\,dx) \cap L^1((R,\infty), x^{\nu+1/2}\,dx)$.
\end{theorem}

\begin{proof}
The proof is straightforward with the bound for $\varphi_{\nu}$. We will split $\R^+$ into two pieces: $(0,R)$ and $(R,\infty)$. For the former, we will use continuity to get a bound for $\varphi_{\nu}$, and for the latter, we will use its asymptotic behavior. Let $y\in\R^+$ and $f\in L^1((0,R), x^{2\nu+1}\,dx) \cap L^1((R,\infty), x^{\nu+1/2}\,dx)$ where $R > 1$ is chosen as in \Cref{lem:asymptotics_of_jnu}, then
\begin{align*}
|\mathcal{F}_{\nu}f(y)| & \le \int_{\R^+} |\varphi_{\nu}(xy)| |f(x)| x^{2\nu+1}\,dx \\
&= \int_0^R |\varphi_{\nu}(xy)| |f(x)| x^{2\nu+1}\,dx + \int_R^{\infty} |\varphi_{\nu}(xy)| |f(x)| x^{2\nu+1}\,dx \\
& \le \int_0^R M_1 |f(x)| x^{2\nu+1}\,dx + \int_R^{\infty} M_2 (xy)^{-\nu-1/2} |f(x)| x^{2\nu+1}\,dx \\
&= M_1 \int_0^R |f(x)| x^{2\nu+1}\,dx + M_2 y^{-\nu-1/2} \int_R^{\infty} |f(x)| x^{\nu+1/2}\,dx \\
&< \infty
\end{align*}
\end{proof}

\begin{corollary}
When $-\frac{1}{2} < \nu$, $\mathcal{F}_{\nu}$ can be defined on $L^1(\R^+, x^{2\nu+1}\,dx)$.
\end{corollary}

\begin{proof}
Let $-\frac{1}{2} < \nu$, then $\nu + \frac{1}{2} < 2\nu + 1$ and for $x > 1$, $x^{\nu+1/2} < x^{2\nu+1}$. Let $R > 1$ be as in \Cref{lem:asymptotics_of_jnu} and $f\in L^1(\R^+, x^{2\nu+1}\,dx)$, then
\begin{align*}
\int_0^R |f(x)| x^{2\nu+1}\,dx + \int_R^{\infty} |f(x)| x^{\nu+1/2}\,dx & \le \int_0^R |f(x)| x^{2\nu+1}\,dx + \int_R^{\infty} |f(x)| x^{2\nu+1}\,dx \\
&= \int_{\R^+} |f(x)| x^{2\nu+1}\,dx \\
&< \infty
\end{align*}

\noindent Thus $f\in L^1((0,R), x^{2\nu+1}\,dx) \cap L^1((R,\infty), x^{\nu+1/2}\,dx)$, and its Fourier--Bessel transform exists as an integral transform.
\end{proof}

When $-\frac{1}{2} < \nu$ and $R > 1$, we have the inclusion of $L^1((R,\infty), x^{2\nu+1}\,dx)$ in $L^1((R,\infty), x^{\nu+1/2}\,dx)$ which supplies the proof and greatly simplifies the classification for the domain of $\mathcal{F}_{\nu}$. However, this is a stricter subset than the domain presented in \Cref{thm:domain_of_fb}, and even the domain presented therein may be a stricter subset than the maximal domain for $\mathcal{F}_{\nu}$ as we used a somewhat crude upper bound for $\varphi_{\nu}$. This is a marked difference from Fourier theory: the Fourier transform of a function $f$ is defined \emph{if and only if} $f\in L^1(\R,dx)$ as $|e^{-ixy}| \equiv 1$. Determining the maximal domain for integral transforms in general is a difficult problem and requires a detailed analysis of the integral kernel. In practice, it is better to have a simpler---if non-optimal---prescription for the domain of integral operators.

While we can have a simpler domain in the case of $-\frac{1}{2} < \nu$, this fails if $\nu < -\frac{1}{2}$ as the inclusion $L^1((R,\infty), x^{2\nu+1}\,dx)$ in $L^1((R,\infty), x^{\nu+1/2}\,dx)$ fails, so we cannot obtain a simpler domain. This leads to the following definition for the domain of the Fourier--Bessel transform as an integral transform.

\begin{definition}
If $-1 < \nu < -\frac{1}{2}$, then $\dom(\mathcal{F}_{\nu}) = L^1((0,R), x^{2\nu+1}\,dx) \cap L^1((R,\infty), x^{\nu+1/2}\,dx)$, where $R$ is some number much larger than $1$. If instead $-\frac{1}{2} < \nu$, then $\dom(\mathcal{F}_{\nu}) = L^1(\R^+, x^{2\nu+1}\,dx)$.
\end{definition}

\section{Properties of the Fourier--Bessel Transform} \label{sec:4}

We will explore the $L^1$ theory of the Fourier--Bessel transform in two separate settings: $-\frac{1}{2} < \nu$ and $-1 < \nu < -\frac{1}{2}$. Again, we can omit discussion of the case that $\nu = -\frac{1}{2}$ as this is simply the Fourier cosine transform. We shall see that the theory is fundamentally very different between $-\frac{1}{2} < \nu$ and $-1 < \nu < -\frac{1}{2}$. While the $L^1$ theory of the Fourier--Bessel transform is fairly clean in the $-\frac{1}{2} < \nu$ setting, it is plagued with challenges in the $-1 < \nu < -\frac{1}{2}$ setting. To date, there has not been a successful generalization of translation in this setting due in part to the growth behavior $\varphi_{\nu}$ has if $-1 < \nu < -\frac{1}{2}$. Moreover, the Riemann--Lebesgue lemma fails if $-1 < \nu < -\frac{1}{2}$. A further confounding factor is that the measure $x^{2\nu+1}\,dx$ is singular at the origin when $-1 < \nu < -\frac{1}{2}$, constraining the regularity of the functions in the domain, and decays at infinity, allowing for functions of very slow decay to be in the domain.

Upon a perusal of the literature, one can see that the Fourier--Bessel transform is not studied when $\nu \le -1$. The Fourier--Bessel transform becomes fairly ill-posed for such $\nu$ as the measure induces a non-integrable singularity at $0$ so that $\chi_{(0,1)}$ is not in the domain of a Fourier--Bessel transform that one might envision for $\nu \le -1$, likewise for the Gaussian for a similar reason so that even the direct $L^2$ theory explored in \cite[Ch. 9]{Akhiezer} is a nonstarter.

As the Fourier--Bessel kernel does not enjoy the group-theoretic properties that the Fourier kernel does, some of the key features of the Fourier transform do not have immediate analogues in the Fourier--Bessel setting. Specifically, the nature of translations in the Fourier--Bessel setting is not straightforward and requires a very different approach and detailed analysis \cite{Delsarte,Levitan_bessel} which leads into the study of hypergroups \cite{Levitan_translation, Rosler_hypergroups}. We omit this discussion for brevity but point the reader to the aforementioned works for further reading. Of particular note is the requirement that $-\frac{1}{2} \le \nu$ in Levitan's original work.

Translations, and therefore modulations, aside, the Fourier--Bessel transform does share some of the same general features of the Fourier transform: there is a dilation property, it has a differential structure, and there is an associated Riemann--Lebesgue lemma when $-\frac{1}{2} < \nu$. These are the focus of this section.

\begin{definition}
Let $\alpha > 0$ and $-1 < \nu$. The dilation operator, $\D_{\nu,\alpha} : L^1(\R^+,x^{2\nu+1}\,dx) \to L^1(\R^+,x^{2\nu+1}\,dx)$ is given by
$$ \D_{\nu,\alpha} f(x) = \alpha^{\nu+1} f(\alpha x). $$
\end{definition}

Just as with the usual dilation operator on $\R$, $\D_{\nu,\alpha}$ has the following group properties: $\D_{\nu,\alpha}\D_{\nu,\beta} = \D_{\nu,\alpha\beta}$ and $\D_{\nu,\alpha}^{-1} = \D_{\nu,\alpha^{-1}}$. These can be shown via direct computation. This leads into our next lemma.

\begin{lemma}
For $f\in L^1(\R^+,x^{2\nu+1}\,dx)$, we have that $\mathcal{F}_{\nu} \D_{\nu,\alpha}f = \D_{\nu,\alpha^{-1}} \mathcal{F}_{\nu} f = \D_{\nu,\alpha}^{-1} \mathcal{F}_{\nu} f$. Furthermore, if $f\in L^1(\R^+,x^{2\nu+1}\,dx)\cap L^2(\R^+,x^{2\nu+1}\,dx)$, then $\langle \D_{\nu,\alpha} f, \D_{\nu,\alpha} f\rangle = \langle f,f\rangle$ so that $\D_{\nu,\alpha}$ extends to a unitary on $L^2(\R^+,x^{2\nu+1}\,dx)$.
\end{lemma}

In addition to the dilation structure of the Fourier--Bessel transform, there is a differential structure to it analogous to the Fourier transform by way of the operator $\Delta_{\nu}$ as $\varphi_{\nu}$ is an eigenfunction of $\Delta_{\nu}$. There is a marked difference between the Fourier and Fourier--Bessel settings: there is no first order differential operator theory for the Fourier--Bessel transform as $\varphi_{\nu}$ is an eigenfunction of a second order differential operator, not an elementary first order differential operator.

\begin{theorem} \label{thm:fb_of_delta}
Let $-\frac{1}{2} < \nu$. If $f\in C^2(\R^+)$ such that it and its derivatives are continuous at $0$ and $f,f',\Delta_{\nu}f\in L^1(\R^+, x^{2\nu+1}\,dx)$, then $\mathcal{F}_{\nu}(\Delta_{\nu} f)(y) = -y^2 \mathcal{F}_{\nu} f(y)$.
\end{theorem}

\begin{proof}
Since $\Delta_{\nu}f \in L^1(\R^+,x^{2\nu+1}\,dx)$, its Fourier--Bessel transform exists. Employing integration by parts, we have
\begin{align*}
\int_{\R^+} \varphi_{\nu}(xy) \Delta_{\nu} f(x) x^{2\nu+1}\,dx &= \int_{\R^+} \varphi_{\nu}(xy) \frac{d}{dx} x^{2\nu+1}\frac{d}{dx} f(x)\,dx \\
&= x^{2\nu+1} \varphi_{\nu}(xy) f'(x)\bigg|_0^{\infty} - \int_{\R^+} \frac{d}{dx} (\varphi_{\nu}(xy)) f'(x) x^{2\nu+1}\,dx
\end{align*}

Evaluating $x^{2\nu+1} \varphi_{\nu}(xy) f'(x)$ at $x= 0$ gives $0$ as $f'(0^+)$ exists and $\varphi_{\nu}(0)$ exists. From the boundedness of $\varphi_{\nu}'$, the integral on the right side is well-defined as $f'\in L^1(\R^+, x^{2\nu+1}\,dx)$. This in conjunction with the existence of the integral on the left side allows us to conclude that $\displaystyle \lim_{x\to 0} x^{2\nu+1} \varphi_{\nu}(xy) f'(x)$ exists. Suppose that the limit is nonzero and, without loss of generality, positive. Then for some large enough $R$, $x^{2\nu+1} \varphi_{\nu}(xy) f'(x) \ge \frac{L}{2}$ for all $x > R$. Thus
$$ \int_R^{\infty} x^{2\nu+1} \varphi_{\nu}(xy) f'(x)\,dx \ge \int_R^{\infty} \frac{L}{2} = \infty $$

\noindent However,
$$ \int_R^{\infty} |x^{2\nu+1} \varphi_{\nu}(xy) f'(x)|\,dx \le \varphi_{\nu}(0) \int_R^{\infty} |f'(x)| x^{2\nu+1}\,dx < \infty $$

\noindent which is a contradiction. Thus, the limit must be $0$, and so
$$ \int_{\R^+} \varphi_{\nu}(xy) \Delta_{\nu} f(x) x^{2\nu+1}\,dx = -\int_{\R^+} \frac{d}{dx} (\varphi_{\nu}(xy)) f'(x) x^{2\nu+1}\,dx. $$

\noindent Similar analysis then gives
$$ \int_{\R^+} \varphi_{\nu}(xy) \Delta_{\nu} f(x) x^{2\nu+1}\,dx = \int_{\R^+} \Delta_{\nu}(\varphi_{\nu}(xy)) f(x) x^{2\nu+1}\,dx = -y^2 \int_{\R^+} \varphi_{\nu}(xy) f(x) x^{2\nu+1}\,dx, $$

\noindent or equivalently, $\mathcal{F}_{\nu}(\Delta_{\nu} f)(y) = -y^2 \mathcal{F}_{\nu} f(y)$.
\end{proof}

This proof is fairly similar to the proof in the Fourier case with some minor but crucial changes due to the weighted measure. It is not clear that this result holds as cleanly when $-1 < \nu < -\frac{1}{2}$ due to the unboundedness of $\varphi_{\nu}$ and the boundary terms at infinity in the integration by parts steps not necessarily being $0$. A similar result to the above holds involving the opposite arrangements of $\mathcal{F}_{\nu}$ and $\Delta_{\nu}$.

\begin{theorem} \label{thm:fb_of_x2}
Let $-\frac{1}{2} < \nu$. If $f,x^2 f\in L^1(\R^+, x^{2\nu+1}\,dx)$, then $\Delta_{\nu} \mathcal{F}_{\nu} f(y) = -\mathcal{F}_{\nu}(x^2 f)(y)$.
\end{theorem}

\begin{proof}
We first remark that if $f, x^2f \in L^1(\R^+, x^{2\nu+1}\,dx$, we get for free that $xf \in L^1(\R^+, x^{2\nu+1}\,dx)$.
\begin{align*}
\int_{\R^+} x|f(x)| x^{2\nu+1}\,dx &= \int_0^1 x |f(x)| x^{2\nu+1}\,dx + \int_1^{\infty} x |f(x)| x^{2\nu+1}\,dx \\
& \le \int_0^1 |f(x)| x^{2\nu+1}\,dx + \int_1^{\infty} x^2 |f(x)|^2 x^{2\nu+1}\,dx
\end{align*}

\noindent The last two integrals are finite by assumption, and so $xf\in L^1(\R^+, x^{2\nu+1}\,dx)$. We proceed to prove the theorem by first considering only the first derivative for which it is critical that $xf\in L^1(\R^+, x^{2\nu+1}\,dx)$.
$$ \frac{d}{dy} \int_{\R^+} \varphi_{\nu}(xy) f(x) x^{2\nu+1}\,dx = \lim_{h\to 0} \int_{\R^+} \frac{\varphi_{\nu}(x(y+h))-\varphi_{\nu}(xy)}{h} f(x) x^{2\nu+1}\,dx. $$

From the mean value theorem, $\frac{\varphi_{\nu}(x(y+h))-\varphi_{\nu}(xy)}{h} = x \varphi_{\nu}'(xy^*)$ for some $y^*\in (y,y+h)$. Boundedness of $\varphi_{\nu}'$ then gives us that the integrand is uniformly bounded by a multiple of $xf(x) x^{2\nu+1}$ which is integrable, so Lebesgue dominated convergence allows us to bring the limit inside the integral:
$$ \frac{d}{dy} \int_{\R^+} \varphi_{\nu}(xy) f(x) x^{2\nu+1}\,dx = \int_{\R^+} \frac{\partial}{\partial y} (\varphi_{\nu}(xy)) f(x) x^{2\nu+1}\,dx. $$

\noindent Proceeding similarly, we have
\begin{align*}
\bigg(y^{-2\nu-1} \frac{d}{dy} y^{2\nu+1} \frac{d}{dy}\bigg) \int_{\R^+} \varphi_{\nu}(xy) f(x) x^{2\nu+1} \,dx &= \int_{\R^+} \bigg(y^{-2\nu-1} \frac{\partial}{\partial y} y^{2\nu+1} \frac{\partial}{\partial y}\bigg)(\varphi_{\nu}(xy)) f(x) x^{2\nu+1}\,dx \\
&= \int_{\R^+} -x^2 \varphi_{\nu}(xy) f(x) x^{2\nu+1}\,dx
\end{align*}

\noindent and $\Delta_{\nu} \mathcal{F}_{\nu}f = -\mathcal{F}_{\nu}(x^2 f)$ as desired.
\end{proof}

The above two theorems are very closely tied to the spectral theory of $\Delta_{\nu}$ on $L^2(\R^+,x^{2\nu+1}\,dx)$. We will see later that the Fourier--Bessel transform extends to a unitary, so it exactly diagonalizes $-\Delta_{\nu}$ into a multiplication operator which is the foundation of the spectral theorem for (essentially) self-adjoint operators.

\subsection{The Riemann--Lebesgue Lemma in the \texorpdfstring{$-\frac{1}{2} < \nu$}{-1/2 < ν} Setting} \label{sec:4_sub:1}

Perhaps the most important feature of the $L^1$ theory of the Fourier transform is the Riemann--Lebesgue lemma. The Riemann--Lebesgue lemma states that the Fourier transform maps $L^1(\R,dx)$ into $C_0(\R)$, the continuous functions that tend to $0$ at infinity. This then allows for estimates of asymptotic behavior for the Fourier transform of functions whose first $k$ derivatives are all in $L^1(\R)$ which plays an important role in the $L^2$ theory of the Fourier transform. The Fourier--Bessel transform also has a Riemann--Lebesgue lemma that is very analogous to the usual Fourier Riemann--Lebesgue lemma. The proof requires some technical lemmas first.

\begin{lemma} \label{lem:RL_Linfinity}
If $-\frac{1}{2} < \nu$ and $f\in L^1(\R^+,x^{2\nu+1}\,dx)$, then $\|\mathcal{F}_{\nu}f\|_{\infty} \le \varphi_{\nu}(0) \|f\|$, i.e. the Fourier--Bessel transform maps into $L^{\infty}(\R^+)$.
\end{lemma}

\begin{proof}
Recall from \Cref{cor:phin_bounded} that $|\varphi_{\nu}(x)| \le \varphi_{\nu}(0)$, therefore
\begin{align*}
|\mathcal{F}_{\nu} f(y)| & \le \int_{\R^+} |\varphi_{\nu}(xy)| |f(x)| x^{2\nu+1}\,dx \\
& \le \varphi_{\nu}(0) \int_{\R^+} |f(x)| x^{2\nu+1}\,dx \\
&= \varphi_{\nu}(0)\|f\|
\end{align*}

\noindent As this is a uniform bound on $\mathcal{F}_{\nu}f$ independent of $y$, the result follows accordingly.
\end{proof}

\begin{lemma} \label{lem:RL_continuity}
If $-\frac{1}{2} < \nu$ and $f\in L^1(\R^+,x^{2\nu+1}\,dx)$, then $\mathcal{F}_{\nu} f$ is continuous.
\end{lemma}

\begin{proof}
To prove this, we rely on Lebesgue dominated convergence and the boundedness and continuity of $\varphi_{\nu}$. Let $(y_n) \subseteq \R^+$ be a sequence converging to $y\in\R^+$. We wish to show that $\mathcal{F}_{\nu}f(y_n)$ converges to $\mathcal{F}_{\nu}f(y)$.

By the boundedness of $\varphi_{\nu}$ per \Cref{cor:phin_bounded}, $|\varphi_{\nu}(x y_n) f(x)| x^{2\nu+1} \le \varphi_{\nu}(0) |f(x)| x^{2\nu+1}$ for all $n$ so that $|\varphi_{\nu}(x y_n) f(x)| x^{2\nu+1}$ is dominated by $\varphi_{\nu}(0) |f(x)| x^{2\nu+1}$ which is integrable by assumption. Continuity of $\varphi_{\nu}$ and Lebesgue dominated converge then give that
\begin{align*}
\lim_{n\to\infty} \mathcal{F}_{\nu}f(y_n) &= \lim_{n\to\infty} \int_{\R^+} \varphi_{\nu}(xy_n) f(x) x^{2\nu+1}\,dx \\
&= \int_{\R^+} \lim_{n\to\infty} \big(\varphi_{\nu}(x y_n) f(x) x^{2\nu+1}\big)\,dx \\
&= \int_{\R^+} \varphi_{\nu}(xy) f(x) x^{2\nu+1}\,dx \\
&= \mathcal{F}_{\nu}f(y)
\end{align*}

\noindent and so $\mathcal{F}_{\nu}f$ is continuous.
\end{proof}

\begin{lemma} \label{lem:RL_density}
If $-\frac{1}{2} < \nu$ and $f\in L^1(\R^+,x^{2\nu+1}\,dx)$, then $\mathcal{F}_{\nu} f$ tends to $0$ at infinity.
\end{lemma}

\begin{proof}
To prove this, we will use a density argument. The characteristic functions $\chi_{(a,b)}$ with $0\le a < b < \infty$ are linearly dense in $L^1(\R^+,x^{2\nu+1}\,dx)$, so if we prove the Riemann--Lebesgue lemma for these functions, much of the argument will have been taken care of already. We can simplify our analysis even further and only consider the characteristic functions $\chi_{(0,a)}$ as $\chi_{(a,b)} = \chi_{(0,b)} - \chi_{(0,a)}$ and extend by linearity. We can simplify further still and consider $\chi_{(0,1)}$ because the Fourier--Bessel transform plays nicely with dilations.

Evaluating $\mathcal{F}_{\nu}\chi_{(0,1)}$, from \cite[5.1.1]{Watson} we have for $-1 < \nu$
\begin{align*}
\mathcal{F}_{\nu}\chi_{(0,1)}(y) &= \int_{\R^+} \varphi_{\nu}(xy) \chi_{(0,1)}(x)x^{2\nu+1}\,dx \\
&= y^{-\nu} \int_0^1 x^{\nu+1} J_{\nu}(xy)\,dx \\
&= y^{-\nu-1} J_{\nu+1}(y)
\end{align*}

\noindent From the asymptotics of the Bessel function per \Cref{lem:asymptotics_of_jnu}, this will tend to $0$ at infinity. By linearity, the Fourier--Bessel transform of any step function will decay to $0$.

Now we wish to extend to all of $L^1(\R^+,x^{2\nu+1}\,dx)$. Central to this argument is that $|\varphi_{\nu}(x)| \le \varphi_{\nu}(0)$. Let $f\in L^1(\R^+,x^{2\nu+1}\,dx)$, $\varepsilon > 0$, and $g$ be a step function such that $\|f-g\| < \frac{\varepsilon}{\varphi_{\nu}(0)}$ whose existence can be guaranteed by the density of the step functions in $L^1(\R^+,x^{2\nu+1}\,dx)$. Then
\begin{equation*}
\mathcal{F}_{\nu}(f-g)(y) = \int_{\R^+} \varphi_{\nu}(xy)(f(x)-g(x))x^{2\nu+1}\,dx
\end{equation*}

\noindent and so we get
\begin{align*}
|\mathcal{F}_{\nu}(f-g)(y)| & \le \int_{\R^+} \big|\varphi_{\nu}(xy) (f(x)-g(x))|x^{2\nu+1}\,dx \\
& \le \varphi_{\nu}(0) \int_{\R^+} |f(x)-g(x)|x^{2\nu+1}\,dx \\
&= \varphi_{\nu}(0)\|f-g\| \\
&< \varepsilon.
\end{align*}

\noindent Therefore we have that $\mathcal{F}_{\nu}f$ and $\mathcal{F}_{\nu}g$ are arbitrarily close to each other for any $y$. As $\mathcal{F}_{\nu}g$ tends to $0$ at infinity, so too must $\mathcal{F}_{\nu}f$, completing the proof.
\end{proof}

\begin{theorem}[Riemann--Lebesgue Lemma]
$\mathcal{F}_{\nu}: L^1(\R^+, x^{2\nu+1}\,dx)\to C_0(\R^+)$.
\end{theorem}

The proof of the Riemann--Lebesgue lemma in the Fourier--Bessel setting is a synthesis of the preceding lemmas.

\subsection{The Riemann--Lebesgue Lemma in the  \texorpdfstring{$-1 < \nu < -\frac{1}{2}$}{-1 < ν < -1/2} Setting} \label{sec:4_sub:2}

There are multiple failure points in an attempt to prove the existence of a Riemann--Lebesgue lemma for the $-1 < \nu < -\frac{1}{2}$ setting. In \Cref{lem:RL_Linfinity}, the boundedness of $\varphi_{\nu}$ when $-\frac{1}{2} < \nu$ is critical, but $\varphi_{\nu}$ grows at infinity when $-1 < \nu < -\frac{1}{2}$. Thus there is no guarantee that the Fourier--Bessel transform maps into $L^{\infty}(\R^+)$ which imperils the Riemann--Lebesgue lemma. Again, \Cref{lem:RL_continuity} suffers similarly. The biggest failure point occurs in \Cref{lem:RL_density}: the density argument collapses as $\varphi_{\nu}$ is not bounded. So while the Fourier--Bessel transform of a step function decays at infinity, this cannot be easily bootstrapped into an argument about the full domain of $\mathcal{F}_{\nu}$. Indeed, the Riemann--Lebesgue lemma fails if $-1 < \nu < -\frac{1}{2}$ as the next theorem shows.

\begin{theorem}
If $-1 < \nu < -\frac{1}{2}$, $\mathcal{F}_{\nu}$ does not have the Riemann--Lebesgue property.
\end{theorem}

\begin{proof}
To demonstrate that the Riemann--Lebesgue lemma does not hold, we simply supply an example for when it fails. Consider $f\in L^1((0,R), x^{2\nu+1}\,dx)\cap L^1((R,\infty), x^{\nu+1/2}\,dx)$, where $R > 1$ is as in \Cref{lem:asymptotics_of_jnu}, defined as
\begin{equation}
f(x) = \frac{2^{\nu+1}\sqrt{\pi}}{\Gamma\big(-\nu-\frac{1}{2}\big)} (1-x^2)^{-\nu-\frac{3}{2}} \chi_{[0,1]}(x),
\end{equation}

\noindent then $\mathcal{F}_{\nu}f(y) = \cos(y)$ \cite[6.567.1]{gradshteyn2007} and does not tend to $0$ as $y\to\infty$.
\end{proof}

This example was initially found via \texttt{Mathematica} by brute forcing a distributional Fourier--Bessel transform of $g(y) = \cos(y)$. This function was chosen as it is bounded and therefore not difficult to work with distributionally but does not tend to $0$ at infinity. The hope was that its inverse would live in $\dom(\mathcal{F}_{\nu})$.

The Riemann--Lebesgue lemma supplies many of the proofs for the $L^2$ theory of the Fourier transform and will also for the Fourier--Bessel transform as we shall see. The failure of the Riemann--Lebesgue lemma for the Fourier--Bessel transform indicates that the $L^1$ to $L^2$ theory is not as clean when $-1 < \nu < -\frac{1}{2}$. The Riemann--Lebesgue lemma can be rescued for $-1 < \nu < -\frac{1}{2}$ for a stricter subspace than the domain of the Fourier--Bessel transform. Specifically, as $\varphi_{\nu}$ has at most polynomial growth, functions with sufficient regularity will have Fourier--Bessel transforms that decay at infinity. See \Cref{sec:7_sub:1} for a related analysis.

\section{A Richer \texorpdfstring{$L^2$}{L2} Theory for Integral Transforms} \label{sec:5}

With the $L^1$ theory of the Fourier--Bessel transform explored, we wish to develop its $L^2$ theory. As such, we must restrict our attention to $L^1(\R^+,x^{2\nu+1}\,dx) \cap L^2(\R^+,x^{2\nu+1}\,dx)$. This space is trivially dense in $L^2(\R^+,x^{2\nu+1}\,dx)$ as it contains the step functions, however proofs are still challenging in this setting unlike in the Fourier case as the integral kernels $\varphi_{\nu}$ do not have a group structure. Instead, we look to nice subspaces of $L^1(\R^+,x^{2\nu+1}\,dx) \cap L^2(\R^+,x^{2\nu+1}\,dx)$ for which proofs are simpler. Prior to that, we establish a general theorem regarding integral operators. First, we make a general remark about operator theory and extensions of operators. The proof is a standard functional analytic argument.

\begin{theorem}
Let $T$ be a densely defined and closed (or closable) operator on a Hilbert space $\mathfrak{H}$. If $T \phi_n = \lambda_n \phi_n$, $\{\phi_n\}$ is a basis for $\mathfrak{H}$, and the $\lambda_n$ are bounded, then $T$ can be extended abstractly to a bounded operator $\widetilde{T}$ on $\mathfrak{H}$ with $\widetilde{T}\mid_{\dom{T}} \equiv T$ and $\|\widetilde{T}\| = \operatorname{sup}_n |\lambda_n|$.
\end{theorem}

This result is critical to the theory of integral transforms and bounded operators in general as they often appear in a concrete form that applies only to a dense subset of a Hilbert space. However this has a major caveat: the extension $\widetilde{T}$ is an abstract operator that is only understood via its action on the basis and extended linearly. Meaning, the extension may have lost characteristics from the original operator. Specifically, if the initial operator was an integral transform, it is possible, if not likely, that its extension may not be realized as an integral transform on the entire Hilbert space.

Specifically in the context of integral transforms, we typically define the integral transform on an $L^1$ space then restrict our attention to an intersection with an $L^2$ space, often with the same measure. Proving $L^2$ results about the integral transform on the intersection of the $L^1$ and $L^2$ space may be challenging, especially if there is no Riemann--Lebesgue lemma to fall back on. For instance, $L^2$ norm-preservation of an integral transform is sometimes only shown concretely on a known basis of eigenfunctions and then extended abstractly to the entire Hilbert space by density arguments \cite[Ch. 9]{Akhiezer}. This however does not guarantee that the integral transform defined on the intersection of the $L^1$ and $L^2$ spaces is norm-preserving as the abstract operator has washed away knowledge of the integral transform on the intersection of the $L^1$ and $L^2$ spaces in favor of an abstract operator. The next theorem rectifies this discrepancy.

\begin{theorem} \label{thm:integral_operator}
Let $\mathcal{S}$ be a dense subspace of $L^1(X,d\mu)\cap L^2(X, d\mu)$ in the $L^2$ norm, $\{\phi_n\}\subseteq \mathcal{S}$ be an orthonormal basis for $L^2(X, d\mu)$ satisfying $\overline{\phi_n(x)} = e^{i\theta_n} \phi_n(x)$ for some $\{\theta_n\}\subseteq \R$, and $T:\mathcal{S}\to L^2(X, d\mu)$ be an integral transform on $\mathcal{S}$. If $T \phi_n = \lambda_n \phi_n$, where $|\lambda_n| = 1$ for all $n$, and $T$ satisfies the Fubini property on $\mathcal{S}$, i.e. for all $f,g\in\mathcal{S}$,
\begin{equation}
\int_X Tf(x) g(x)\,d\mu(x) = \int_X f(x) Tg(x)\,d\mu(x), \label{eq:fubini}
\end{equation}

\noindent then $T$ is in fact a norm-preserving integral transform on $\mathcal{S}$ and extends to a unitary operator on $L^2(X,d\mu)$.
\end{theorem}

\begin{proof}
Let $f\in \mathcal{S}$, then $Tf\in L^2(X, d\mu)$ by assumption and so $\langle Tf, \phi_n\rangle$ exists and is finite for all $n$. Making use of the assumptions that the eigenfunctions $\phi_n$ satisfy $\overline{\phi_n(x)} = e^{i\theta_n} \phi_n(x)$ and $T$ satisfies the Fubini property, we have
\begin{align*}
\langle Tf, \phi_n\rangle &= \int_X Tf(x) \overline{\phi_n(x)}\,d\mu(x) \\
&= e^{i\theta_n} \int_X Tf(x) \phi_n(x)\,d\mu(x) \\
&= e^{i\theta_n} \int_X f(x) T\phi_n(x)\,d\mu(x) \\
&= \lambda_n e^{i\theta_n} \int_X f(x) \phi_n(x)\,d\mu(x) \\
&= \lambda_n \langle f, \phi_n\rangle
\end{align*}

\noindent Because $|\lambda_n| = 1$, $\|Tf\|^2 = \sum_n |\langle Tf, \phi_n\rangle|^2 = \sum_n |\lambda_n \langle f, \phi_n\rangle|^2 = \sum_n |\langle f, \phi_n\rangle|^2 = \|f\|^2$ and the norm-preservation of $T$ as an integral transform on $\operatorname{span}\{\phi_n\}$ can be extended so $T$ that is norm-preserving on all of $\mathcal{S}$ as an integral transform.

As $\mathcal{S}$ is dense in $L^2(X,d\mu)$ and $T$ is norm-preserving on $\mathcal{S}$, $T$ can be extended to a norm-preserving map on all of $L^2(X,d\mu)$. Moreover, its extension is in fact surjective by virtue of $\{\phi_n\}$ being a basis of eigenfunctions for $T$ and thus $\operatorname{span}\{T\phi_n\} = \operatorname{span}\{\phi_n\}$ is dense in $L^2(X,d\mu)$ so that the closure of the range of the extension of $T$ is all of $L^2(X,d\mu)$. As the extension of $T$ is norm-preserving (and thereby injective) and surjective, the extension of $T$ must be unitary.
\end{proof}

This theorem has a direct corollary which is useful in the study of integral transforms. The proof is simple and omitted.

\begin{corollary}
Let the conditions for \Cref{thm:integral_operator} hold with the added constraints that $T:\mathcal{S}\to\mathcal{S}$ and $T \phi_n = \pm \phi_n$ so that $T^2 \phi_n = \phi_n$, then $T^2 = I$ on $\mathcal{S}$.
\end{corollary}

Such operators $T$ satisfying the Fubini property are complex symmetric operators \cite{Garcia_2014}, where the conjugation is the usual complex conjugate, which have a large and growing literature in recent years.

\begin{remark}
We did not explicitly need that $T$ was defined as an integral transform in the above statements. In fact, the arguments show that any concretely defined operator satisfying the conditions for the theorem can be extended as that concrete operator beyond eigenfunctions to the larger dense subspace. However, it is of interest to know when unitary operators arising from integral transforms have the norm-preservation property as an integral transform rather than abstractly extended operator, and the Fubini property is often satisfied by integral transforms when restricted to appropriate subspaces. Moreover, the condition $\overline{\phi_n(x)} = e^{i\theta_n} \phi_n(x)$ is not very onerous as this is satisfied for any purely real or imaginary eigenfunctions but provides for other cases. As such, the above results hold for concretely-defined operators $T$ satisfying the Fubini property.
\end{remark}

For integral transforms that play nicely with dilations, Akhiezer's trick can be used to develop an infinite family of eigenfunctions if one knows just one eigenfunction. However, it can be difficult to determine that this is a basis. Indeed, Akhiezer appeals to complex analysis and effectively the Fourier transform to argue that a function that is orthogonal to all of the eigenfunctions of the Fourier--Bessel transform must be $0$. Instead, one can use the dilation property and the knowledge of only one eigenfunction to generate many more. Prior to that, we need a technical lemma about uncountable linearly independent sets of elements in separable Hilbert spaces.

\begin{lemma} \label{lem:countable_dense_set}
Let $\mathcal{H}$ be a separable Hilbert space and $\Lambda$ be an uncountable set. If $\{f_{\lambda}:\lambda\in\Lambda\}$ is a linearly independent set and $\operatorname{span}\{f_{\lambda}:\lambda\in\Lambda\}$ is dense in $\mathcal{H}$, then there is a countable $\Lambda' \subseteq \Lambda$ such that $\operatorname{span}\{f_{\lambda}:\lambda\in\Lambda'\}$ is dense in $\mathcal{H}$, i.e. a countable spanning set can be extracted.
\end{lemma}

\begin{proof}
Let $\varepsilon > 0$ and $f\in \mathcal{H}$. As $\mathcal{H}$ is separable, it has a countable orthonormal basis $\{\phi_n:n\in\mathbb{N}\}$. Moreover, since $\operatorname{span}\{f_{\lambda}:\lambda\in\Lambda\}$ is dense in $\mathcal{H}$, there exist $l_k\in\mathbb{N}$ and $b_{k_n}\in\C$ such that
\begin{equation*}
\bigg\|\phi_k - \sum_{n=1}^{l_k} b_{k_n} f_{k_n}\bigg\| < \frac{\varepsilon}{2^{k+1} \|f\|}.
\end{equation*}

\noindent Because the $\{\phi_n\}$ form an orthonormal basis for $\mathcal{H}$, there exists $c_n$ and $M$ such that
\begin{equation*}
\bigg\|f - \sum_{m=1}^M c_m \phi_m\bigg\| < \frac{\varepsilon}{2}.
\end{equation*}

\noindent From the above, we have that
\begin{align*}
\bigg\|f - \sum_{m=1}^M c_m \sum_{n=1}^{l_m} b_{m_n} f_{m_n} \bigg\| & \le \bigg\|f - \sum_{m=1}^M c_m \phi_m\bigg\| + \bigg\|\sum_{m=1}^M c_m \phi_m - \sum_{m=1}^M c_m \sum_{n=1}^{l_m} b_{m_n} f_{m_n} \bigg\| \\
& \le \frac{\varepsilon}{2} + \sum_{m=1}^M |c_m| \bigg\|\phi_m - \sum_{n=1}^{l_m} b_{m_n} f_{m_n}\bigg\| \\
& < \frac{\varepsilon}{2} + \sum_{m=1}^M \|f\| \frac{\varepsilon}{2^{k+1}\|f\|} \\
&= \varepsilon
\end{align*}

\noindent so that the set $\{f_{k_n}\}$ is a linearly independent set with dense span in $\mathcal{H}$ and is moreover countable as $\{f_{k_n}\}$ is finite for a fixed $k$.
\end{proof}

\begin{lemma}
Let $T:L^2(\X,d\mu)\to L^2(\X,d\mu)$, where $\D_{\alpha} \X \subseteq \X$ for all $\alpha > 0$, such that $T \D_{\alpha} = \D_{\alpha}^{-1}T$ and $f$ be an eigenfunction of $T$. If there exists an infinite collection $\{\alpha_j\}$ such that $\operatorname{span}\{\D_{\alpha_j} f\}$ is dense in $L^2(\X,d\mu)$, then $T$ admits a basis of eigenfunctions.
\end{lemma}

\begin{proof}
Suppose $Tf = \lambda f$, then $T(\D_{\alpha_j}f + \D_{\alpha_j^{-1}}f) = \lambda (\D_{\alpha_j}f + \D_{\alpha_j^{-1}}f)$ and likewise $T(\D_{\alpha_j}f - \D_{\alpha_j^{-1}}f) = -\lambda (\D_{\alpha_j}f - \D_{\alpha_j^{-1}}f)$ so that both $\D_{\alpha_j}f + \D_{\alpha_j^{-1}}f$ and $\D_{\alpha_j}f - \D_{\alpha_j^{-1}}f$ are eigenfunctions. As there are dilations $\{\D_{\alpha_j}f\}$ of $f$ that are linearly independent, we can pick eigenfunctions $\D_{\alpha_j}f \pm \D_{\alpha_j^{-1}}f$ that are linearly independent by \Cref{lem:countable_dense_set}. The set of functions $\D_{\alpha_j}f \pm \D_{\alpha_j^{-1}}f$ are also a basis for $L^2(\X,d\mu)$ as their span is exactly $\operatorname{span}\{\D_{\alpha_j} f\}$. We also get for free that the eigenvalues of $T$ are $\pm \lambda$.
\end{proof}

\begin{remark}
If $\operatorname{span}\{\D_{\alpha} f\}$ is not dense and instead forms a subspace with a nonzero complement in $L^2(\X,d\mu)$, we can repeat this process in the orthogonal complement and proceed inductively.
\end{remark}

To employ these results, we must make a judicious choice of space on which to define our integral transform. In standard Fourier theory, two such choices are typically made: on $L^1(\R,dx)\cap L^2(\R,dx)$ and on the Schwartz space. The Fourier--Bessel transform is also amenable to both approaches with some distinct changes from the usual Fourier setting.

\section{Eigenfunctions of the Fourier--Bessel Transform} \label{sec:6}

Central to the results in Section 5 is the existence of eigenfunctions of the integral transform and that the integral transform obeys the Fubini property \eqref{eq:fubini}. A well-known result is that the Gaussian is an eigenfunction of the Fourier--Bessel transform, exactly in accordance with the Fourier transform.

\begin{lemma}
If $-1 < \nu$, then the Gaussian $e^{-x^2/2}$ is in the domain of $\mathcal{F}_{\nu}$ and is an eigenfunction of $\mathcal{F}_{\nu}$ with eigenvalue $1$.
\end{lemma}

That $e^{-x^2/2}$ is in the domain of $\mathcal{F}_{\nu}$ can be checked by direct computation of its $L^1$ norm. Furthermore, that it is an eigenfunction can be shown via Fubini's theorem and the series representation for $J_{\nu}$. This integral is well-known in the literature, see \cite[6.629.4]{gradshteyn2007}. To apply \Cref{thm:integral_operator} and \Cref{lem:countable_dense_set} to conclude that the Fourier--Bessel transform as an integral transform is norm-preserving, we need that the dilations of the Gaussian are linearly dense in $L^2(\R^+, x^{2\nu+1}\,dx)$. We have a more general result by Hamburger \cite{Hamburger}, see also \cite[Ch. 9]{Akhiezer}.

\begin{lemma} \label{lem:poly_exp_density}
Let $-1 < \nu$ and $f\in L^2(\R^+, x^{2\nu+1}\,dx)$. If
\begin{equation}
\int_{\R^+} f(x) x^{2n} e^{-x^2/2} x^{2\nu+1}\,dx = 0
\end{equation}

\noindent for all $n = 0, 1, 2, \ldots$, then $f\equiv 0$ a.e., i.e. if $f$ is orthogonal to $x^{2n} e^{-x^2/2}$ for all $n$, then $f\equiv 0$.
\end{lemma}

This lemma guarantees that the dilations of the Gaussian are linearly dense in $L^2(\R^+,x^{2\nu+1}\,dx)$. By \Cref{lem:countable_dense_set}, there is a basis of dilations of the Gaussians for $L^2(\R^+,x^{2\nu+1}\,dx)$ with eigenvalues $\pm 1$. This lemma can be extended further for functions $f$ that are not in $L^2(\R^+, x^{2\nu+1}\,dx)$ themselves, but the above theorem statement suffices for our purposes.

\Cref{lem:countable_dense_set} is non-constructive by nature and so we only know that there is \emph{some} collection of dilations of the Gaussian that form a basis. In what follows, we develop concrete orthogonal eigenfunctions to the Fourier--Bessel transform that form a basis for $L^2(\R^+,x^{2\nu+1}\,dx)$. Note that non-orthogonal eigenfunctions can be established by way of Akhiezer's trick \cite[Ch. 9]{Akhiezer} which explicitly relies on the dilation property of the Fourier--Bessel transform.

\begin{definition} \label{def:phi_n}
Let $-1 < \nu$ and $n\in \N_0$. Define $\phi_n$ by
\begin{equation}
\phi_n(x) = \frac{(-1)^n}{\sqrt{2^{4n-1} n! \Gamma(n+\nu+1)}} e^{x^2/2} \Delta_{\nu}^n e^{-x^2}.
\end{equation}
\end{definition}

\begin{lemma} \label{lem:phin_recurrence}
The $\phi_n$ satisfy the following recurrence relation:
\begin{equation}
\bigg(-\Delta_{\nu} - x^2 + 2x\frac{d}{dx} + 2\nu+2\bigg)\phi_n(x) = 4 \sqrt{n+1} \sqrt{n+\nu+1} \phi_{n+1}(x).
\end{equation}
\end{lemma}

\begin{proof}
For convenience, let $\mu_n = \frac{(-1)^n}{\sqrt{2^{4n-1} n! \Gamma(n+\nu+1)}}$ and $R_{\nu} = -\Delta_{\nu} - x^2 + 2x\frac{d}{dx} + 2\nu+2$. Note that there is a product (or Leibniz) rule for $\Delta_{\nu}$:
\begin{equation}
\Delta_{\nu} (fg) = (\Delta_{\nu} f)g + 2f'g' + f(\Delta_{\nu}g).
\end{equation}

\noindent Applying $R_{\nu}$ to $\phi_n$ per its definition in \Cref{def:phi_n} in conjunction with the product rule for $\Delta_{\nu}$ gives
\begin{equation}
R_{\nu} \phi_n = \mu_n e^{x^2/2} \Delta_{\nu}^{n+1} e^{-x^2} = 4 \sqrt{n+1} \sqrt{n+\nu+1} \phi_{n+1}
\end{equation}

\noindent as desired.
\end{proof}

As per \Cref{thm:integral_operator}, we need that the $\phi_n$ form a basis for $L^2(\R^+, x^{2\nu+1}\,dx)$. This is not immediately guaranteed from the definition. For instance, if we defined $\phi_n$ instead by
\begin{equation*}
\phi_n(x) = \frac{(-1)^n}{\sqrt{2^{4n-1} n! \Gamma(n+\nu+1)}} e^{x^2/2} \Delta_{\nu}^{n+1} e^{-x^2},
\end{equation*}

\noindent where we replaced $\Delta_{\nu}^n$ with $\Delta_{\nu}^{n+1}$, this would no longer be true as the Gaussian $e^{-x^2/2}$ which is in $L^2(\R^+, x^{2\nu+1}\,dx)$ would be orthogonal to all of the $\phi_n$ and thus not in their span.

\begin{theorem}
Let $-1 < \nu$, $n\in \N_0$, and $\phi_n$ be defined as in \Cref{def:phi_n}, then $\phi_n \in \dom(\mathcal{F}_{\nu})$, $\phi_n\in L^2(\R^+,x^{2\nu+1}\,dx)$, the $\phi_n$ are an orthonormal basis for $L^2(\R^+, x^{2\nu+1}\,dx)$, and $\mathcal{F}_{\nu} \phi_n = (-1)^n \phi_n$.
\end{theorem}

\begin{proof}
A simple induction shows that $\phi_n$ is a $2n$-degree polynomial multiplying $e^{-x^2/2}$ and so are in $\dom(\mathcal{F}_{\nu})$ and $L^2(\R^+,x^{2\nu+1}\,dx)$ for all $-1 < \nu$. We will focus on proving that they are an orthonormal basis for $L^2(\R^+, x^{2\nu+1}\,dx)$. Without loss of generality, assume $m < n$. We wish to show that $\langle \phi_m, \phi_n\rangle = 0$. For simplicity, we ignore the normalization factors.
\begin{align*}
\langle \phi_m, \phi_n \rangle & \propto \int_{\R^+} (e^{x^2/2} \Delta_{\nu}^m e^{-x^2})\overline{(e^{x^2/2} \Delta_{\nu}^n e^{-x^2})} x^{2\nu+1}\,dx \\
&= \int_{\R^+} (e^{x^2} \Delta_{\nu}^m e^{-x^2}) (\Delta_{\nu}^n e^{-x^2} x^{2\nu+1}\,dx) \\
&= \int_{\R^+} \Delta_{\nu}^{m+1} (e^{x^2} \Delta_{\nu}^m e^{-x^2}) (\Delta_{\nu}^{n-m-1} e^{-x^2} x^{2\nu+1}\,dx)
\end{align*}

\noindent The integration by parts is justified here as the functions decay to zero rapidly and are continuous and the boundary terms go to $0$ at $0$. As $e^{x^2} \Delta_{\nu}^m e^{-x^2}$ is a polynomial of degree $2m$, $\Delta_{\nu}^{m+1} (e^{x^2} \Delta_{\nu}^m e^{-x^2}) = 0$, proving orthogonality.

Let $f\in L^2(\R^+, x^{2\nu+1}\,dx)$ be orthogonal to all of the $\phi_n$. Since the $\phi_n$ are $2n$-degree polynomials multiplying $e^{-x^2/2}$, taking linear combinations gives that $f$ is orthogonal to $x^{2n} e^{-x^2/2}$ for all $n$. By \Cref{lem:poly_exp_density}, this implies that $f\equiv 0$ and so the $\phi_n$ form a complete basis.

To prove that the $\phi_n$ are orthonormal, we use the completeness in conjunction with the operator $R_{\nu}$. We proceed by induction on $n$. A direct computation shows that $\|\phi_0\| = 1$. Suppose that $\|\phi_n\| = 1$. We wish to show that $\|\phi_{n+1}\| = 1$. Employing integration by parts, we have the following identity
\begin{equation}
\int_{\R^+} R_{\nu} \phi_n(x) \phi_m(x) x^{2\nu+1}\,dx = \int_{\R^+} \phi_n(x)\bigg(-\Delta_{\nu} - x^2 - 2x\frac{d}{dx} - 2\nu-2\bigg)\phi_m(x) x^{2\nu+1}\,dx.
\end{equation}

\noindent Let $L_{\nu} = -\Delta_{\nu} - x^2 - 2x\frac{d}{dx} - 2\nu-2$ so that the previous identity reads
\begin{equation}
\langle R_{\nu} \phi_n, \phi_m\rangle = \langle \phi_n, L_{\nu} \phi_m\rangle
\end{equation}

\noindent for any $n,m$. Taking $m = n+1$ and writing $L_{\nu}\phi_{n+1} = \sum_m c_m \phi_m$ and taking an inner product with $\phi_k$, we have
\begin{equation*}
c_k = \langle L_{\nu} \phi_{n+1}, \phi_k\rangle = \langle \phi_{n+1}, R_{\nu} \phi_k\rangle = \langle \phi_{n+1}, 4 \sqrt{k+1} \sqrt{k+\nu+1} \phi_{k+1}\rangle = 4 \sqrt{k+1} \sqrt{k+\nu+1} \delta_{n+1,k+1},
\end{equation*}

\noindent and so $c_k = 4 \sqrt{n+1} \sqrt{n+\nu+1} \delta_{n,m}$, i.e. $L_{\nu} \phi_{n+1} = 4 \sqrt{n+1} \sqrt{n+\nu+1} \phi_n$. Thus,
\begin{equation*}
4 \sqrt{n+1} \sqrt{n+\nu+1} \langle \phi_{n+1}, \phi_{n+1} \rangle = \langle R_{\nu} \phi_n, \phi_{n+1} \rangle = \langle \phi_n, L\phi_{n+1} \rangle = 4 \sqrt{n+1} \sqrt{n+\nu+1} \langle \phi_n, \phi_n \rangle,
\end{equation*}

\noindent and thus $\|\phi_{n+1}\| = 1$.

To show that the $\phi_n$ are eigenfunctions of the Fourier--Bessel transform with eigenvalue $(-1)^n$, we proceed again by induction. As noted previously, $\phi_0$ is an eigenfunction of $\mathcal{F}_{\nu}$ with eigenvalue $1$. By \Cref{lem:phin_recurrence}, $\phi_n \propto R_{\nu}^n \phi_0$. To prove the result, we have the secondary result that for $f\in \operatorname{span}\{\phi_n\}$,
\begin{equation}
\mathcal{F}_{\nu} R_{\nu}f = -R_{\nu} \mathcal{F}_{\nu}f.
\end{equation}

\noindent This can be achieved quickly by applications of integration by parts as well as differentiation under the integral, both of which are justified because $f$ is even and has $0$ derivative at $0$ and are of rapid decay. A key feature of this argument is that $-\Delta_{\nu}(\varphi_{\nu}(xy)) = y^2 \varphi_{\nu}(xy)$. Thus, up to an overall multiplicative factor depending on $n$ (indicated by the proportionalities below),
\begin{equation}
\mathcal{F}_{\nu} \phi_n \propto \mathcal{F}_{\nu} R_{\nu}^n \phi_0 = (-1)^n R_{\nu}^n \mathcal{F}_{\nu} \phi_0 = (-1)^n R_{\nu}^n \phi_0 \propto (-1)^n \phi_n
\end{equation}

\noindent as desired.
\end{proof}

\section{The \texorpdfstring{$L^2$}{L2} Theory of the Fourier--Bessel Transform} \label{sec:7}

\subsection{The Fourier-Bessel Transform on \texorpdfstring{$L^1(\R^+,x^{2\nu+1}\, dx)\cap L^2(\R^+,x^{2\nu+1}\,dx)$}{L1∩L2}} \label{sec:7_sub:1}

In this section, we assume that $-\frac{1}{2} < \nu$. Similar to the usual Fourier transform, to expedite the arguments for the Fourier--Bessel transform on $L^1(\R^+,x^{2\nu+1}\, dx)\cap L^2(\R^+,x^{2\nu+1}\,dx)$, we restrict our attention to functions with sufficient regularity. Our goal is to make the integration by parts steps later rigorous which informs the choice for the domain. In this section, we shall consider the subspace
$$ \mathcal{S} = \left\{f\in C^{2k}(\R^+)\cap L^1(\R^+,x^{2\nu+1}\, dx)\cap L^2(\R^+,x^{2\nu+1}\,dx) \;\left|\; \begin{aligned} & \Delta_{\nu}f, \ldots, \Delta_{\nu}^k f \in L^1(\R^+,x^{2\nu+1}\,dx), \\
& \frac{d}{dx}f, \ldots, \frac{d}{dx}\Delta_{\nu}^{k-1} f \in L^1(\R^+,x^{2\nu+1}\,dx) \end{aligned}\right.\right\}, $$

\noindent where $k = 2(\nu+2)$ if $\nu$ is an integer and $k = 2\lceil\nu+1\rceil$ if $\nu$ is non-integer. Note $k$ is even.

We could also require that $f',f''\in L^2(\R^+,x^{2\nu+1}\,dx)$, but this is an unnecessary assumption for what follows. Notice that we have a potentially large number of derivatives in the definition of $\S$ to offset the growth from the measure to ensure that the Fourier--Bessel transform will be in $L^2(\R^+,x^{2\nu+1}\,dx)$. Here, we assume that $f$ and its derivatives are all bounded at $0$ which is important for the later analysis. Alternatively, we could assume that $f$ and its derivatives all extend continuously to $0$.

$\mathcal{S}$ is necessarily nonempty as it includes the Gaussian $e^{-x^2/2}$ and its dilations, $\phi_n$ for all $n$, and all compactly supported smooth functions supported away from $0$, each of which can be checked via direct computation. $\mathcal{S}$ is much larger than this of course, but this is sufficient for our purposes. Furthermore, $\mathcal{S}$ is dense in $L^2(\R^+,x^{2\nu+1}\,dx)$ due to its containment of these functions. A simple application of Fubini's theorem shows that the Fourier--Bessel transform satisfies the Fubini property on $\S$ for $-\frac{1}{2} < \nu$ as $\varphi_{\nu}$ is bounded.

From the Riemann--Lebesgue lemma, we know that the Fourier--Bessel transform of a function in $L^1(\R^+,x^{2\nu+1}\,dx)$ is in $C_0(\R^+)$, but we have no control on the decay of the Fourier--Bessel transform of the function \emph{a priori}. Thus we cannot conclude without more work or assumptions that the Fourier--Bessel transform of functions in $L^1(\R^+,x^{2\nu+1}\, dx)\cap L^2(\R^+,x^{2\nu+1}\,dx)$ are necessarily in $L^2(\R^+,y^{2\nu+1}\,dy)$ themselves. The extra derivative assumptions will provide more decay for the Fourier--Bessel transform and therefore guarantee containment in $L^2(\R^+,y^{2\nu+1}\,dy)$ as shown in the next theorem.

\begin{theorem}
If $f\in\mathcal{S}$, then $\mathcal{F}_{\nu} f\in L^2(\R^+,y^{2\nu+1}\,dy)$, $\mathcal{F}_{\nu}$ is norm-preserving, and $\mathcal{F}_{\nu}$ extends to a unitary on $L^2(\R^+,x^{2\nu+1}\,dx)$.
\end{theorem}

\begin{proof}
$f$ satisfies the conditions for \Cref{thm:fb_of_delta}, and a modest modification of the result shows more generally that
\begin{equation}
\int_{\R^+} \varphi_{\nu}(xy) \Delta_{\nu}^{k/2} f(x) x^{2\nu+1}\,dx = (-y^2)^{k/2} \int_{\R^+} \varphi_{\nu}(xy) f(x) x^{2\nu+1}\,dx.
\end{equation}

By the Riemann--Lebesgue lemma, we know that the Fourier--Bessel transform of $\Delta_{\nu}^{k/2} f$ is bounded, and from the final equality, we can conclude that the Fourier--Bessel transform of $f$ decays at least as fast as $y^{-k}$. Thus
$$ \int_{\R^+} |\mathcal{F}_{\nu}f(y)|^2 y^{2\nu+1}\,dy = \int_0^R |\mathcal{F}_{\nu}f(y)|^2 y^{2\nu+1}\,dy + \int_R^{\infty} |\mathcal{F}_{\nu}f(y)|^2 y^{2\nu+1}\,dy $$

\noindent The first integral on the right is finite by virtue of the continuity of $\mathcal{F}_{\nu}f$. The second is finite since $|\mathcal{F}_{\nu}f(y)|^2 \le M y^{-2k}$ and so the integrand is $y^{-2k+2\nu+1}$. Since $k = 2(\nu+2)$ if $\nu$ is an integer or $2\lceil\nu+1\rceil$ if $\nu$ is non-integer, $\displaystyle \int_R^{\infty} y^{-2k+2\nu+1}\,dy$ is finite. Thus $\mathcal{F}_{\nu} f\in L^2(\R^+, y^{2\nu+1}\,dy)$ as desired.

Thus by \Cref{thm:integral_operator}, the Fourier--Bessel transform is norm-preserving on $\S$ and extends to a unitary on $L^2(\R^+,x^{2\nu+1}\,dx)$. Furthermore, since $\mathcal{F}_{\nu}^2 \phi_n = \phi_n$ for all $n$, the extension of $\mathcal{F}_{\nu}^2$ to $L^2(\R^+,x^{2\nu+1}\,dx)$ is the identity.
\end{proof}

\subsection{A Schwartz Space for the Fourier--Bessel Transform} \label{sec:7_sub:2}

In contrast with the Fourier setting, the Fourier--Bessel kernel $\varphi_{\nu}$ is not an eigenfunction of a linear first order differential operator involving only elementary functions. However, it is an eigenfunction of a first order differential-difference equation \cite{Dunkl,roslervoit} which provides an entry point for Dunkl theory which will not be explored here.

$\varphi_{\nu}$ is however an eigenfunction of the second order differential operator $\Delta_{\nu}$ which suggests that the Schwartz space associated to the Fourier--Bessel transform is a bit different than that of the Fourier transform. To ensure that $\Delta_{\nu}$ defines a symmetric operator which can then be extended to a self-adjoint operator, a careful analysis shows that the Schwartz space should be the restriction of the even Schwartz space functions to the half-line, else there are boundary terms at $0^+$ to contend with. This motivates the following definition.

\begin{definition}
The half-line Schwartz space $\S^+(\R^+)$ is given by
\begin{equation}
\S^+(\R^+) = \left\{f|_{\R^+} \;\left|\; f\in \S(\R), f(-x) = f(x)\right.\right\},
\end{equation}

\noindent where $\S(\R)$ denotes the usual Schwartz space \cite[p. 336]{Conway}.
\end{definition}

In either case that $-1 < \nu < -\frac{1}{2}$ or $-\frac{1}{2} < \nu$, the Fourier--Bessel transform is well-defined as an integral transform on $\S^+(\R^+)$. Moreover, $\S^+(\R^+)$ is dense in $L^2(\R^+,x^{2\nu+1}\,dx)$ as it contains the Gaussian and its dilations, $\phi_n$ for all $n$, and the compactly supported smooth functions supported away from $0$. One can endow $\S^+(\R^+)$ with a topology, but we will not do this here as it is a major digression as $\S^+(\R^+)$ is not the primary focus of the paper.

Additionally, the Fourier--Bessel transform satisfies the Fubini property on $\S^+(\R^+)$ which can be shown via direct computation. When $-\frac{1}{2} < \nu$, the argument is nearly identical to that as in Section 7.1 as $\S^+(\R^+) \subseteq L^1(\R^+,x^{2\nu+1}\,dx)$. However, when $-1 < \nu < -\frac{1}{2}$, the exponential decay offsets the growth from the measure and integrals may be interchanged by Fubini's theorem. $\S^+(\R^+)$ plays nicely with $x^2$, $\Delta_{\nu}$, and the Fourier--Bessel transform as the next results show.

\begin{lemma}
If $f\in \S^+(\R^+)$, then $x^2 f,\Delta_{\nu} f\in \S^+(\R^+)$.
\end{lemma}

\begin{proof}
Since multiplying an even Schwartz function by $x^2$ yields another even Schwartz function, it follows that if $f\in\S^+(\R^+)$, then $x^2 f\in \S^+(\R^+)$. Likewise, if $f$ is an even Schwartz function, so too is $f''$. As $\Delta_{\nu}f = f'' + \frac{2\nu+1}{x} f'$, we need only to show that $\frac{f'}{x}$ can be associated to an even Schwartz function.

Clearly, $\frac{f'}{x}$ is undefined at $x = 0$, but the discontinuity is removable as $f$ is even. Define $g:\R\to\R$ by
\begin{equation}
g(x) = \begin{cases}\displaystyle  \frac{f'(x)}{x}, & x \neq 0 \\ f''(0), & x = 0 \end{cases}
\end{equation}

\noindent $g$ is continuous on $\R$ and has rapid decay at infinity, inherited from $f$. A similar argument shows that the derivatives of $g$ are all continuous at $0$, and therefore $g$ is smooth on $\R$. Since $g$ is smooth on $\R$, has rapid decay at infinity, and is even on $\R$, it is an even Schwartz function. Thus $\frac{f'}{x}$ is the half-line restriction of an even Schwartz function and so lies in $\S^+(\R^+)$, giving that $\Delta_{\nu} f \in \S^+(\R^+)$ as desired.
\end{proof}

\begin{theorem}
If $f\in \S_{\nu}^+(\R^+)$, then $\mathcal{F}_{\nu} (x^2 f) = -\Delta_{\nu} \mathcal{F}_{\nu} f$ and $\mathcal{F}_{\nu} (\Delta_{\nu} f)(y) = -y^2 \mathcal{F}_{\nu} f(y)$.
\end{theorem}

\begin{proof}
This is a direct consequence of \Cref{thm:fb_of_delta} and \Cref{thm:fb_of_x2}.
\end{proof}

\begin{corollary}
If $f\in\S^+(\R^+)$, then $\mathcal{F}_{\nu}f \in\S^+(\R^+)$ and $\mathcal{F}_{\nu}^2 = I$ on $\mathcal{S}^+_{\nu}(\R^+)$.
\end{corollary}

\begin{proof}
This proof is nearly a direct synthesis of the previous results. From the Riemann--Lebesgue lemma in conjunction with the fact that $\Delta_{\nu}^k f\in \S^+(\R^+)$ for any $f\in \S^+(\R^+)$, the Fourier--Bessel transform of any half-line Schwartz function has exponential decay. Similarly, its Fourier--Bessel transform will be smooth as $x^{2n}f\in \S^+(\R^+)$. Thus, for any $f\in \S^+(\R^+)$, $\mathcal{F}_{\nu}f$ will be the restriction of a Schwartz function and evenness is obtained via the evenness of $\varphi_{\nu}$ so that $\mathcal{F}_{\nu}(\S^+(\R^+)) \subseteq \S^+(\R^+)$. Moreover, since $\mathcal{F}_{\nu}^2 \phi_n = \phi_n$, $\mathcal{F}_{\nu}^2 = I$ on $\mathcal{S}_{\nu}^+(\R^+)$ and $\mathcal{F}_{\nu}$ is an automorphism of $\S^+(\R^+)$.
\end{proof}

In effect, the half-line Schwartz space is invariant under $x^2$, $\Delta_{\nu}$, and the Fourier--Bessel transform, and $-\Delta_{\nu}$ has symbol $y^2$ under the Fourier--Bessel transform, all in near perfect parallel with the Fourier transform and the whole-line Schwartz space. Not only does the Fourier--Bessel transform map $\S_{\nu}^+(\R^+)$ to itself, it also satisfies the Fubini property on $\S_{\nu}^+(\R^+)$ which can be checked quickly. Thus by \Cref{thm:integral_operator}, the Fourier--Bessel transform is norm-preserving on $\S^+(\R^+)$ and extends to a unitary on $L^2(\R^+,x^{2\nu+1}\,dx)$.

\section{An Uncertainty Principle for the Fourier--Bessel Transform} \label{sec:8}

The Fourier uncertainty principle is fundamental in the mathematical sciences as it places a limitation on how well-localized a function and its Fourier transform may be. In terms of signal processing, it imposes a limit on how well one can identify when a band of frequencies occurred in a signal: the stricter the frequency band, the less knowledge there is of when the frequencies occurred in the signal. Mathematically, it may be presented as
\begin{equation}
\frac{1}{4} \le \bigg(\int_{\R} (x-a)^2 |f(x)|^2\,dx\bigg) \bigg(\int_{\R} (y-b)^2 |\mathcal{F}f(y)|^2\,dy\bigg),
\end{equation}

\noindent where $a,b\in\R$, for an $L^2$-normalized function $f$. When $a = \langle x\rangle$, where $\langle x\rangle = \langle xf,f\rangle$, the integral
$$ \int_{\R} (x-a)^2 |f(x)|^2\,dx $$

\noindent represents the variance of $x$ with respect to the probability density $|f(x)|^2$ and measures the spread in $f$ away from its mean, similarly for $b$. The integrals represent the ``uncertainty'' in measurements: the smaller the uncertainty, the more the content of $f$ is concentrated around its mean so that we have a good understanding of the values of $f$ and vice versa.

A simple variational argument shows that the uncertainty product for a fixed $f$ is minimized when $a = \langle x\rangle$ and likewise for $b$. The uncertainty principle places a constraint on how closely $f$ and $\mathcal{F}f$ can be concentrated around their respective means. There are numerous proofs for the Fourier uncertainty principle, a short collection of which can be found in the survey by Folland \cite{Folland}. Key to many of these proofs are the group theoretic properties of the Fourier kernel and Lebesgue measure. The weighted measure in the Fourier--Bessel setting already supplies an obstruction to generalizing these techniques, further exacerbated by the fact that the $\varphi_{\nu}$ do not have group properties but rather hypergroup properties \cite{Levitan_translation}.

Nevertheless, one could inquire about the nature of
\begin{equation}
J[f] = \bigg(\int_{\R^+} (x-a)^2 |f(x)|^2 x^{2\nu+1}\,dx\bigg) \bigg(\int_{\R^+} (y-b)^2 |\mathcal{F}_{\nu}f(y)|^2 y^{2\nu+1}\,dy\bigg)
\end{equation}

\noindent for the Fourier--Bessel transform, where $a,b\in\R^+$. While the usual Fourier techniques do not apply, a variational approach could be applied to determine the local extrema of the above functional. Doing the variational derivative of this functional subject to the constraint that $\|f\| = 1$ gives the following integro-differential equation
\begin{equation}
-\alpha \Delta_{\nu} f - 2b\alpha \mathcal{F}_{\nu}^{-1}(y\mathcal{F}_{\nu}f) + \beta (x-a)^2 f = \lambda f.
\end{equation}

Note the appearance of the term $\mathcal{F}_{\nu}^{-1}(y\mathcal{F}_{\nu}f)$. This is a major departure from the Fourier setting and occurs because the Fourier--Bessel transform does not play nicely with $x$ like the Fourier transform but rather with $x^2$. This cannot be simplified further, and as such, most authors consider only the case when $a = 0 = b$ so that this term will vanish and to enforce symmetry in the uncertainty product.

Starting from $a = 0 = b$ gives the functional
\begin{equation}
J[f] = \bigg(\int_{\R^+} x^2 |f(x)|^2 x^{2\nu+1}\,dx\bigg) \bigg(\int_{\R^+} y^2 |\mathcal{F}_{\nu}f(y)|^2 y^{2\nu+1}\,dy\bigg).
\end{equation}

\noindent If $f\in L^2(\R^+,x^{2\nu+1}\,dx)$ such that $x^2 f\in L^2(\R^+,x^{2\nu+1}\,dx)$ and $y^2 \mathcal{F}_{\nu}f \in L^2(\R^+,y^{2\nu+1}\,dy)$, the above functional is well-defined and we may explore the nature of the uncertainty product.

\begin{theorem}
Let $-\frac{1}{2} < \nu$ and $f\in\S$, where $\S$ is as defined in Section 7.1, such that $\|f\| = 1$, $x^2 f\in L^2(\R^+,x^{2\nu+1}\,dx)$, and $\Delta_{\nu} f\in\S$, then $J$ has a lower bound of $(\nu+1)^2$ which is obtained by dilations of the Gaussian.
\end{theorem}

\begin{proof}
Since $\|f\| = 1$, we have
\begin{equation}
1 = \int_{\R^+} |f(x)|^2 x^{2\nu+1}\,dx = \frac{1}{2(\nu+1)} x^{2\nu+2} f(x)\overline{f(x)} \bigg|_0^{\infty} - \frac{1}{\nu+1} \int_{\R^+} x \operatorname{Re}(f'(x) \overline{f(x)}) x^{2\nu+1}\,dx.
\end{equation}

\noindent The boundary terms vanish at infinity by virtue of $f$ being bounded at $0$ and $x^{2\nu+2} \to 0$ at $0$ and the $L^2$ condition on $f$. Employing Cauchy--Schwarz, splitting the weight equally onto both parts, and another integration by parts yields
\begin{equation}
1 \le \frac{1}{(\nu+1)^2} \bigg(\int_{\R^+} x^2 |f(x)|^2 x^{2\nu+1}\,dx\bigg) \bigg(\int_{\R^+} f(x) \overline{(-\Delta_{\nu} f(x))} x^{2\nu+1}\,dx\bigg).
\end{equation}

\noindent Again, the boundary terms in the integration by parts step go to $0$ because of the $L^2$ condition on $f$.

By virtue of the unitarity of the Fourier--Bessel transform and $\mathcal{F}_{\nu} \Delta_{\nu}f(y) = -y^2 \mathcal{F}_{\nu}f(y)$, this gives exactly that
\begin{equation}
(\nu+1)^2 \le \bigg(\int_{\R^+} x^2 |f(x)|^2 x^{2\nu+1}\,dx\bigg) \bigg(\int_{\R^+} y^2 |\mathcal{F}_{\nu} f(y)|^2 y^{2\nu+1}\,dy\bigg).
\end{equation}

\noindent From the Cauchy--Schwarz step, equality is obtained when $f'(x) = -\lambda x f(x)$ and so $f$ is a Gaussian. Note that the real part of $\lambda$ must be positive for $f\in L^2(\R^+,x^{2\nu+1}\,dx)$.
\end{proof}

When $\nu = -\frac{1}{2}$, this becomes exactly the usual Fourier uncertainty principle. The above analysis can be repeated for $-1 < \nu < -\frac{1}{2}$ when $f\in \S^+(\R^+)$ and is in fact simplified in this case. There are some other generalizations of the Fourier uncertainty principle to the Fourier--Bessel transform, particularly the work of \cite{GHOBBER2011501} regarding the supports of a function and its Fourier--Bessel transform.

Exploring the local extrema the functional $J$ after making use of \Cref{thm:fb_of_delta}, we have the following eigenvalue problem for the local extrema
\begin{equation}
-\alpha \Delta_{\nu} f + \beta x^2 f = \lambda f.
\end{equation}

\noindent This is a quantum harmonic oscillator-like Hamiltonian. The eigenfunctions are exactly $\phi_n$ defined in \Cref{def:phi_n} and correspond to local minimizers for additive uncertainty uncertainty. A global minimizer for the additive minimizer can be established via completeness of the eigenfunctions. Such an approach was taken by De Bruijn \cite{debruijn} in the Fourier setting. The dilation trick employed by Folland \cite{Folland} allows for this to be ported to the multiplicative uncertainty to obtain a minimum multiplicative uncertainty.

As noted above, the weighted measure and lack of group theoretic properties in the Fourier--Bessel setting are a major obstruction to a more faithful analogue of the uncertainty principle for the Fourier--Bessel transform. A further complication is that the Fourier--Bessel transform is best understood as a transform acting on the restrictions of even functions to the half-line. An analogue of the uncertainty principle for the Fourier--Bessel transform on the whole line was explored by \cite{roslervoit} that requires a Fourier sine transform analogue paired with the Fourier--Bessel transform. In this setting, an explicit lower bound is established even when $a$ and $b$ are nonzero. This approach is based on Dunkl theory which is based on a first order differential-difference operator on $\R$. This first order theory streamlines the arguments and supplies the elegance of the approach.

\bibliographystyle{plain}
\bibliography{references}

\begin{thebibliography}{10}

\bibitem{Akhiezer}
N.~I. Akhiezer.
\newblock {\em Lectures on integral transforms}.
\newblock American Mathematical Society, 1988.

\bibitem{debruijn}
N.G. \{Bruijn, de\}.
\newblock Uncertainty principles in fourier analysis.
\newblock In {\em Inequalities : proceedings of a symposium held at Wright-Patterson air force base, Ohio, August 19-27, 1965}, pages 57--71. Academic Press Inc., 1967.

\bibitem{Conway}
J.~B. Conway.
\newblock {\em A Course in Functional Analysis}.
\newblock Springer New York, NY, 2007.

\bibitem{Delsarte}
J.~Delsarte.
\newblock Sur une extension de la formule de {Taylor}.
\newblock {\em Journal de Math\'ematiques Pures et Appliqu\'ees}, 9e s{\'e}rie, 17(1-4):213--231, 1938.

\bibitem{NIST:DLMF}
{\it NIST Digital Library of Mathematical Functions}.
\newblock \url{https://dlmf.nist.gov/}, Release 1.2.4 of 2025-03-15.
\newblock F.~W.~J. Olver, A.~B. {Olde Daalhuis}, D.~W. Lozier, B.~I. Schneider, R.~F. Boisvert, C.~W. Clark, B.~R. Miller, B.~V. Saunders, H.~S. Cohl, and M.~A. McClain, eds.

\bibitem{Dunkl}
C.~F. Dunkl.
\newblock Differential-difference operators associated to reflection groups.
\newblock {\em Transactions of the American Mathematical Society}, 311(1):167--183, 1989.

\bibitem{Folland}
Gerald~B. {Folland} and Alladi {Sitaram}.
\newblock {The uncertainty principle: A mathematical survey}.
\newblock {\em Journal of Fourier Analysis and Applications}, 3(3):207--238, May 1997.

\bibitem{Garcia_2014}
Stephan~Ramon Garcia, Emil Prodan, and Mihai Putinar.
\newblock Mathematical and physical aspects of complex symmetric operators.
\newblock {\em Journal of Physics A: Mathematical and Theoretical}, 47(35):353001, aug 2014.

\bibitem{GHOBBER2011501}
Saifallah Ghobber and Philippe Jaming.
\newblock Strong annihilating pairs for the fourier–bessel transform.
\newblock {\em Journal of Mathematical Analysis and Applications}, 377(2):501--515, 2011.

\bibitem{gradshteyn2007}
I.~S. Gradshteyn and I.~M. Ryzhik.
\newblock {\em Table of integrals, series, and products}.
\newblock Elsevier/Academic Press, Amsterdam, seventh edition, 2007.
\newblock Translated from the Russian, Translation edited and with a preface by Alan Jeffrey and Daniel Zwillinger, With one CD-ROM (Windows, Macintosh and UNIX).

\bibitem{Hamburger}
H.~Hamburger.
\newblock Beitr\"{a}ge zur konvergenztheorie der stieltjesschen kettenbr\"{u}che.
\newblock {\em Mathematische Zeitschrift}, pages 186--222, 1919.

\bibitem{Hankel1875}
Hankel.
\newblock Die fourier'schen reihen und integrale für cylinderfonctionen.
\newblock {\em Mathematische Annalen}, 8:471--494, 1875.

\bibitem{Levitan_bessel}
B.~M. Levitan.
\newblock Expansion in fourier series and integrals with bessel functions.
\newblock {\em Uspekhi Mat. Nauk}, 6:102--143, 1951.

\bibitem{Levitan_translation}
B.~M. Levitan.
\newblock {\em Generalized Translation Operators and Some of Their Applications}.
\newblock Israel Program for Scientific Translations, 1964.

\bibitem{Rosler_hypergroups}
Margit Rösler.
\newblock Bessel-type signed hypergroups on r.
\newblock In {\em Probability measures on groups and related structures XI (Oberwolfach 1994)}, page 292–304. World Scientific, 1995.

\bibitem{roslervoit}
Margit Rösler and Michael Voit.
\newblock An uncertainty principle for hankel transforms.
\newblock {\em Proceedings of the American Mathematical Society}, 127(1):183--194, 1999.

\bibitem{Watson}
George~Neville Watson.
\newblock Bessel functions. (scientific books: A treatise on the theory of bessel functions).
\newblock {\em Science}, 1923.

\end{thebibliography}

\end{document}